\documentclass[leqno]{amsart}

\usepackage{stmaryrd,graphicx}
\usepackage{amssymb,mathrsfs,amsmath,amscd,amsthm,color}
\usepackage{float}
\usepackage{mathabx}
\usepackage[all,cmtip]{xy}
\DeclareMathAlphabet{\mathpzc}{OT1}{pzc}{m}{it}
\usepackage{amsfonts,latexsym,wasysym}

\newcommand{\ui}{I}

\newcommand{\wt}{\widetilde}

\newcommand{\mcc}{\mathcal{C}}

\newcommand{\scra}{\mathscr{A}}
\newcommand{\scrb}{\mathscr{B}}
\newcommand{\scrc}{\mathscr{C}}
\newcommand{\scrd}{\mathscr{D}}

\newcommand{\scrr}{\mathscr{R}}
\newcommand{\scrs}{\mathscr{S}}

\newcommand{\bbh}{\mathbb{H}}
\newcommand{\bbn}{\mathbb{N}}

\newcommand{\bbr}{\mathbb{R}}

\newcommand{\bbz}{\mathbb{Z}}

\newtheorem{theorem}{Theorem}[section]
\newtheorem{lemma}[theorem]{Lemma}

\newtheorem{corollary}[theorem]{Corollary}
\theoremstyle{definition}\newtheorem{definition}[theorem]{Definition}
\newtheorem{example}[theorem]{Example}

\newtheorem{remark}[theorem]{Remark}

\begin{document}
\title{The infinitary $n$-cube shuffle}

\author[J. Brazas]{Jeremy Brazas}
\address{West Chester University\\ Department of Mathematics\\
West Chester, PA 19383, USA}
\email{jbrazas@wcupa.edu}

\subjclass[2010]{ 55Q52 , 55Q20 ,08A65  }
\keywords{higher homotopy group, infinite product, infinitary commutativity, Eckmann-Hilton principle, k-dimensional Hawaiian earring}

\date{\today}

\begin{abstract}
In this paper, we formalize the sense in which higher homotopy groups are ``infinitely commutative." In particular, we both simplify and extend the highly technical procedure, due to Eda and Kawamura, for constructing homotopies that isotopically rearrange infinite configurations of disjoint $n$-cubes within the unit $n$-cube.
\end{abstract}

\maketitle

\section{Introduction}

The classical Eckmann-Hilton Principle \cite{EckHilt} allows one to construct boundary-relative homotopies between $n$-loops, i.e. maps $(I^n,\partial I^n)\to (X,x)$, that realize rearrangements of finitely many disjoint $n$-cubes within $I^n$. As an immediate consequence, all higher homotopy groups are commutative. In general, finite configurations of disjoint $n$-cubes play an important role in homotopy theory as they form the little $n$-cubes operad \cite{BV}, which provides a convenient setting in which one can effectively recognize higher loop space structure \cite{MayGoILS}. In this paper, we develop methods for rearranging infinite configurations of $n$-cubes with application to the higher homotopy groups of Peano continua that admit natural infinite product (i.e. infinitary) operations, e.g. the $n$-dimensional Hawaiian earring $\bbh_n$ studied in \cite{BarrattMilnor,EK00higher,Kawamurasuspensions}.

In \cite{EK00higher}, Eda and Kawamura develop a highly technical sequence of homotopies, which form a procedure for continuously rearranging infinitely many disjoint $n$-cube domains with $I^n$. Their argument is applied to $n$-loops in a shrinking wedge $\wt{\bigvee}_{m}X_m$ of a sequence $\{X_m\}_{m\in\bbn} $ of $(n-1)$-connected spaces with a semilocal contractability condition at the basepoint, $\pi_k(\wt{\bigvee}_{m}X_m)$ is trivial when $1\leq k\leq n-1$ and the inclusion $\wt{\bigvee}_{m}X_m\to \prod_{m}X_m$ induces an isomorphism on $\pi_n$, giving $\pi_n(\wt{\bigvee}_{m}X_m)\cong \prod_{m}\pi_n(X_m)$. A more broadly applicable rearrangement result appears in the proof of \cite[Theorem 2.1]{Kawamurasuspensions}; however, this only applies to one specific configuration of $n$-cubes. Further understanding of the higher homotopy groups of Peano continua will require infinite and broadly applicable analogues of classical homotopy theory techniques such as cellular approximation and homotopy excision. However, even in the classical situation, one must construct homotopies carefully to ensure continuity. For instance, the proof of the standard cellular approximation theorem requires one to construct homotopies, which are the constant homotopy on predetermined subsets of the domain. 

The purpose of this paper is to (1) develop broadly applicable techniques for studying higher homotopy groups that admit geometrically relevant infinite product operations (2) provide a simpler proof of the original infinitary $n$-cube shuffle argument given in the work of Eda and Kawamura, and (3) strengthen the original argument by constructing homotopies that shuffle infinitely many $n$-cubes while remaining constant on a predetermined set of finitely many $n$-cubes whose complement is path connected.

We define an \textit{$n$-domain} to be a non-empty finite collection $\scrr=\{R_k\mid k\in K\}$ of $n$-cubes in $I^n$ (ordered by a finite or infinite set $K\subseteq \bbn$) whose interiors are pairwise disjoint. Given a sequence $\{f_k\}_{k\in K}$ of maps $f_k:(I^n,\partial I^n)\to (X,x)$ such that for every neighborhood $U$ of $x$, all but finitely many of the maps $f_k$ have image lying in $U$, we may form the \textit{$\scrr$-concatenation} $\prod_{\scrr}f_k:(I^n,\partial I^n)\to (X,x)$, which is defined as $f_k$ on $R_k$ and is constant at $x$ otherwise. If $\scrs=\{S_k\mid k\in K\}$ is another $n$-domain, the $\scrs$-concatenation $\prod_{\scrs}f_k$ may be thought of as a rearrangement of the map $\prod_{\scrr}f_k$ by moving each $R_k$ to the new position $S_k$. Our main result, which says that any desired shuffle can be achieved through a continuous homotopy, is the following.

\begin{theorem}[Infinitary $n$-Cube Shuffle]\label{infiniteeckmannhiltontheorem} Given $n$-domains $\mathscr{R}=\{R_k\mid k\in K\}$ and $\mathscr{S}=\{S_k\mid k\in K\}$ and $K$-sequence $\{f_k\}_{k\in K}$ in $\Omega^{n}(X,x)$, we have $\prod_{\mathscr{R}}f_k\simeq\prod_{\mathscr{S}}f_k$ in $X$ by a homotopy with image in $\bigcup_{k\in K}Im(f_k)$. Moreover, if there is a finite set $F\subseteq K$ such that $R_k=S_k$ for all $k\in F$ and $I^n\backslash\bigcup_{k\in F}R_k$ is path connected, then we may choose the homotopy to be the constant homotopy on $\bigcup_{k\in F}R_k$.
\end{theorem}

The difficulty in proving Theorem \ref{infiniteeckmannhiltontheorem} lies in its generality. When $K$ is infinite, the ordinary Eckmann-Hilton Principle no longer applies; one cannot simply rearrange finitely many of the $n$-cubes $R_k$ at a time and expect to continuously perform all required rearrangements. Additionally, both $n$-domains $\scrr$ and $\scrs$ might both be dense in $I^n$ or might even form an infinite cubical decomposition of $I^n$ in the sense that $\bigcup_{k\in K}R_k=\bigcup_{k\in K}S_k=I^n$. Moreover, infinitely many $R_k,S_k$ might have diameter $1$ and it may also be the case that for every $k\in K$, $R_k$ is nowhere near $S_k$ within $I^n$. Hence, one cannot simply perform smaller homotopies for smaller $n$-cubes. The difficulty becomes even more pronounced when one is required to perform the rearrangements while remaining fixed on a finite subset of $\scrr$.

We refer to the first statement of Theoren \ref{infiniteeckmannhiltontheorem} as the \textit{unrestricted case} and the second statement as the \textit{restricted case} since, in the second statement, we must construct the homotopy to be constant on a given finite sub-$n$-domain. Additionally, we will treat the case where $K$ is \textit{finite} separately from the case where $K$ is \textit{infinite}. Hence, we split up Theorem \ref{infiniteeckmannhiltontheorem} naturally into four (non-mutually exclusive) cases: finite unrestricted, finite restricted, infinite unrestricted, and infinite restricted, where the last case is the strongest statement to be proven.

In Section \ref{sectionprelim}, we settle definitions and notation and we also prove the particularly important ``$n$-domain shrinking Lemma" (Lemma \ref{shrinkingcubelemma}), which we make use of repeatedly throughout the paper. In Section \ref{sectionfinitecase}, we prove the unrestricted and restricted finite cases of Theorem \ref{infiniteeckmannhiltontheorem} (see Corollary \ref{finitecasecor}). The finite cases are fairly standard in homotopy theory \cite{MayGoILS}. We include these proofs (1) since they are short, (2) to provide a rigorous foundation for the infinite case using our $n$-domain framework, and (3) to contrast with the infinite case (See Remark \ref{temptingremark}). In Section \ref{sectioninfinitecase}, we prove Lemma \ref{eckhiltontechlemma} from which the infinite unrestricted case of Theorem \ref{infiniteeckmannhiltontheorem} follows immediately. The infinite unrestricted case is essentially the extent of what is achieved in \cite{EK00higher}; our proof utilizes the unrestricted finite case in a way that simplifies the overall argument given in \cite{EK00higher}. We then use both the restricted finite case and unrestricted infinite case to prove the restricted infinite case and complete the proof of Theorem \ref{infiniteeckmannhiltontheorem}.

In Section \ref{sectionconsequences}, we point out some immediate consequences of Theorem \ref{infiniteeckmannhiltontheorem}. For instance, we show that for any sequence $X_1,X_2,X_3,\dots$ of path-connected spaces, the homotopy long exact sequence of the pair $\left(\prod_{m}X_m,\wt{\bigvee}_{m}X_m\right)$ splits naturally in every dimension $n\geq 2$ (Corollary \ref{splittingcorollary}), yielding an isomorphism 
$\pi_n\left(\wt{\bigvee}_{m}X_m\right)\cong \pi_{n+1}\left(\prod_{m}X_m,\wt{\bigvee}_{m}X_m\right)\oplus \prod_{m}\pi_n(X_m)$. After an initial version of this manuscript was written, the author learned of the paper \cite{Kawamurasuspensions} by Kawamura where the above isomorphism appears as one of the main results. We conclude the paper with the brief Remark \ref{operadremark}, which shows how the framework of the current paper provides a natural extension of the little $n$-cubes operad $\mcc_n(j)$ to an infinite term $\mcc_n(\omega)$.


\section{Preliminary definitions and $n$-domain shrinking}\label{sectionprelim}

Throughout this paper, $n\geq 2$ is a fixed integer, $I=[0,1]$ is the unit interval, $(X,x)$ is a path-connected based topological space, and $\Omega^{n}(X,x)$ represents the space of relative maps $(I^n,\partial I^n)\to (X,x)$ (with the compact-open topology) so that $\pi_n(X,x)=\pi_0(\Omega^{n}(X,x))$ is the $n$-th homotopy group. The term \textit{$n$-cube} will refer to subsets of $I^n$ of the form $\prod_{i=1}^{n}[a_i,b_i]$. Given two $n$-cubes $R=\prod_{i=1}^{n}[a_i,b_i]$ and $R'= \prod_{i=1}^{n}[c_i,d_i]$ in $I^n$, let $L_{R,R'}:R\to R'$ denote the canonical homeomorphism, which is increasing and linear in each component. If $f:R\to X$ and $g:R'\to X$ are maps such that $f=g\circ L_{R,R'}$, then we write $f\equiv g$. Given maps $f_1,f_2,\dots, f_m\in \Omega^{n}(X,x)$, the \textit{$m$-fold concatenation} $\prod_{i=1}^{m}f_i=f_1\cdot f_2 \cdots f_m$ is the map $(I^n,\partial I^n)\to (X,x)$ whose restriction to $[(i-1)/m,i/m]\times I^{n-1}$ is $f_i\circ L_{[(i-1)/m,i/m]\times I^{n-1},I^n}$. If $K$ is an ordered set with the order type of $\bbn$, then a sequence of maps $\{f_k\}_{k\in K}$ in $\Omega^n(X,x)$ is \textit{null (at $x$)} if every neighborhood of $x$ contains the image $Im(f_k)$ for all but finitely many $k\in K$.

\begin{definition}
Let $K\subseteq \bbn$. 
\begin{enumerate}
\item An \textit{$n$-domain (over $K$)} is an ordered set $\mathscr{R}=\{R_k\mid k\in K\}$ of $n$-cubes in $I^n$ whose interiors are pairwise disjoint, i.e. $int(R_k)\cap int(R_{k'})=\emptyset$ if $k\neq k'$.
\item A finite $n$-domain $\{R_1,R_2,R_3,\dots, R_m\}$ is a \textit{cubical decomposition} of $I^n$ if $\bigcup_{k=1}^{m}R_k=I^n$.
\item A sequence of maps $\{f_k\}_{k\in K}$ in $\Omega^{n}(X,x)$ is called a \textit{$K$-sequence} if either $K$ is finite or if $K$ is infinite and $\{f_k\}_{k\in K}$ is null at $x$.
\item If $\mathscr{R}$ is an $n$-domain over $K$, the \textit{$\mathscr{R}$-concatenation} of a $K$-sequence $\{f_k\}_{k\in K}$ is the map $\prod_{\mathscr{R}}\{f_k\}_{k\in K}\in \Omega^n(X,x)$ whose restriction to $R_k$ is $f_k\circ L_{R_k,I^n}$ and which maps $I^n\backslash \bigcup_{n}R_n$ to $x$.
\end{enumerate}
\end{definition}

When $K$ is clear from context, we will refer to an $n$-domain $\scrr$ over $K$ simply as an \textit{$n$-domain} and we will denote the $\mathscr{R}$-concatenation by $\prod_{\mathscr{R}}f_k$.

\begin{definition}
If $\mathscr{R}=\{R_k\mid k\in K\}$ and $\mathscr{S}=\{S_k\mid k\in K\}$ are $n$-domains, we call $\mathscr{S}$ an \textit{sub-$n$-domain} of $\mathscr{R}$ if $S_k\subseteq R_k$ for all $k\in K$.
\end{definition}

The next lemma states that we may always replace an $n$-domain with an arbitrary sub-$n$-domain. It provides a crucial technique that we will use repeatedly in the proof of both the finite and infinite cases of Theorem \ref{infiniteeckmannhiltontheorem}.

\begin{lemma}[$n$-domain shrinking]\label{shrinkingcubelemma}
Let $\mathscr{S}=\{S_k\mid k\in K\}$ be any sub-$n$-domain of $\mathscr{R}=\{R_k\mid k\in K\}$. If $\{f_k\}_{k\in K}$ is any $K$-sequence in $\Omega^n(X,x)$, then $\prod_{\mathscr{R}}f_k\simeq \prod_{\mathscr{S}}f_k$ by a homotopy in $\bigcup_{k\in\bbn}Im(f_k)$. Moreover, if $R_k=S_k$, then this homotopy may be chosen to be the constant homotopy on $R_k$.
\end{lemma}

\begin{proof}
For each $k\in K$, let $T_k=\{((1-t)a+tL_{R_k,S_k}(a),t)\in I^n\times I\mid a\in R_k, t\in\ui\}$ be the convex hull of $R_k\times \{0\}\cup S_k\times \{1\}$ in $I^{n+1}$ and note that $T_k\subseteq R_{k}\times I$. Define homotopy $H:I^n\times I\to X$ by \[H(b,s)=\begin{cases}
x , &\text{ if }(b,s)\in I^{n+1}\backslash \bigcup_{k\in K}T_k\\
f_k(L_{R_k,I^n}(a)) &\text{ if }b=(1-s)a+sL_{R_k,S_k}(a)\text{ for }a\in R_k.
\end{cases}\]
Since $H(\partial T_k\backslash (int(R_k)\times \{0\}\cup int(S_k)\times \{1\}))=x$, $H$ is well-defined. In the case that $K$ is infinite, the continuity of $H$ follows from the fact that $Im(H(T_k))=Im(f_k)$ for all $k\in\bbn$. The last statement of the lemma is evident from the construction of $H$.
\end{proof}
\section{The Finite Case}\label{sectionfinitecase}

If $K$ is any set, then $\Sigma_{K}$ will denote the symmetric group of all bijections $K\to K$ and for $m\in\bbn$, $\Sigma_m$ will denote the symmetric group of $\{1,2,\dots ,m\}$.

\begin{lemma}\label{eckmannhiltonlemma}
If $\mathscr{R}=\{R_1,R_2,\dots,R_m\}$ is an $n$-domain, $f_1,f_2,\dots, f_m\in \Omega^n(X,x)$, and $\phi\in \Sigma_m$, then $\prod_{\mathscr{R}}f_k\simeq \prod_{k=1}^{m}f_{\phi(k)}$ by a homotopy in $\bigcup_{k=1}^{m}Im(f_k)$.
\end{lemma}
\begin{proof}
Write $R_k=\prod_{i=1}^{n}J_{i}^{k}$ and find pairwise disjoint sets $[c_1,d_1],[c_2,d_2],\dots,[c_m,d_m]$ in $I$ such that $[c_{k},d_{k}]\subseteq J_{n}^{k}$ for all $k\in\{1,2,\dots,m\}$. Consider the following sequence of finite $n$-domains.
\begin{enumerate}
\item $A_k=\prod_{i=1}^{n-1}J_{i}^{k}\times [c_{k},d_{k}]$ and $\scra=\{A_1,A_2,\dots,A_m\}$,
\item $B_k=\prod_{i=1}^{n-1}I\times [c_{k},d_{k}]$ and $\scrb=\{B_1,B_2,\dots,B_m\}$,
\item $C_k=\left[\frac{\phi^{-1}(k)-1}{m},\frac{\phi^{-1}(k)}{m}\right]\times\prod_{i=2}^{n-1}I\times [c_{k},d_{k}]$ and $\scrc=\{C_1,C_2,\dots,C_m\}$,
\item $D_k=\left[\frac{\phi^{-1}(k)-1}{m},\frac{\phi^{-1}(k)}{m}\right]\times\prod_{i=2}^{n}I$ and $\scrd=\{D_1,D_2,\dots,D_m\}$.
\end{enumerate}
Each one of $\scrr,\scra,\scrb,\scrc,\scrd$ is an $n$-domain where we have pairwise inclusions $R_k\supseteq A_k\subseteq B_k \supseteq C_k\subseteq D_k$ for all $k\in\{1,2,\dots,m\}$. By Lemma \ref{shrinkingcubelemma}, we have
\[\prod_{\mathscr{R}}f_k\simeq \prod_{\mathscr{A}}f_k\simeq \prod_{\mathscr{B}}f_k\simeq \prod_{\mathscr{C}}f_k\simeq \prod_{\mathscr{D}}f_k=\prod_{k=1}^{m}f_{\phi(k)}\]where all homotopies may be chosen to have image in $\bigcup_{k=1}^{m}Im(f_k)$.
\end{proof}

\begin{lemma}\label{twocyclelemma}
Suppose $\scrr=\{R_1,R_2,R_3,\dots, R_m\}$ and $\scrs=\{S_1,S_2,\dots S_m\}$ are $n$-domains such that $R_k=S_k$ for all $k$ except possibly one value $k_0\in \{1,2,\dots ,m\}$. If $f_1,f_2,\dots ,f_m\in \Omega^n(X,x)$, and $F\subseteq \{1,2,\dots ,m\}\backslash\{k_0\}$ is such that $I^n\backslash \bigcup_{k\in F}R_k$ is path-connected, then there is a homotopy $\prod_{\scrr}f_k\simeq \prod_{\scrs}f_{k}$ with image in $\bigcup_{k=1}^{m}Im(f_k)$ and which is the constant homotopy on $\bigcup_{k\in F}R_k$.
\end{lemma}

\begin{proof}
Supposing that $R_{k_0}\neq S_{k_0}$ (for otherwise the constant homotopy may be used), find $n$-cubes $R_{k_0}'\subseteq R_{k_0}$ and $S_{k_0}'\subseteq S_{k_0}$ such that $R_{k_0}'\cap S_{k_0}'=\emptyset$. Applying Lemma \ref{shrinkingcubelemma}, we may replace $R_{k_0}$ with $R_{k_0}'$ in $\scrr$ and $S_{k_0}$ with $S_{k_0}'$ in $\scrs$. Hence, we may assume that $R_{k_0}\cap S_{k_0}=\emptyset$. Pick $r_0\in int(R_{k_0})$ and $s_0\in int(S_{k_0})$. Since $r_0,s_0$ lie in the simply connected open $n$-manifold $(0,1)^n\backslash \bigcup_{k\in F}R_k$, we may find two polygonal arcs $a_0,a_1\subseteq I^n\backslash \bigcup_{k\in F}R_k$ with endpoints $r_0$ and $s_0$, which may be thickened to neighborhoods $M,N\subseteq (0,1)^n\backslash \bigcup_{k\in F}R_k$ respectively so that $M,N$ are homeomorphic to an $n$-disk and $M\cap N$ has two components homeomorphic to an $n$-disk, one of which lies in $int(R_{k_0})$ and the other in $int(S_{k_0})$. Using these arcs and neighborhoods, it is straightforward to construct isotopies $G_0,G_1:I^n\times I\to I^n\backslash \bigcup_{k\in F}R_k$ such that 
\begin{enumerate}
\item $A_{t}=G_0(I^n\times \{t\})$ and $B_{t}=G_1(I^n\times \{t\})$ are $n$-cubes for all $t\in I$,
\item $A_{0}=B_1=R_{k_0}$, $A_{1}=B_0=S_{k_0}$, 
\item $G_0(I^n\times [0,1/3))\subseteq R_{k_0}$, $G_0(I^n\times (1/4,3/4))\subseteq M$, $G_0(I^n\times (2/3,1])\subseteq S_{k_0}$,
\item $G_1(I^n\times [0,1/3))\subseteq S_{k_0}$, $G_1(I^n\times (1/4,3/4))\subseteq N$, $G_1(I^n\times (2/3,1])\subseteq R_{k_0}$.
\end{enumerate}
Effectively, the isotopies $G_0$ and $G_1$ switch the positions of the $n$-cubes $R_{k_0},S_{k_0}$ so that, at any point during the isotopy, the images of these $n$-cubes do not intersect each other nor any of the $n$-cubes $R_{k}=S_k$, $k\in F$. Let $\scrs=\{R_k\mid k\in F\}$ and put $g=\prod_{\scrs}\{f_k\}_{k\in F}$. Noting that $g(R_{k_0})=g(S_{k_0})=x$, the desired homotopy is the map $H:I^n\times I\to X$ defined so that $H(x,t)=f_{k_0}(G_0(x,t))$ if $(x,t)\in A_t\times \{t\}$, $H(x,t)=f_{k_1}(G_1(x,t))$ if $(x,t)\in B_t\times \{t\}$, and $H(x,t)=g(x)$ if $(x,t)\notin T=\bigcup_{t\in I}(A_t\times \{t\}\cup B_t\times \{t\})$. The well-definedness and continuity of $H$ are routine to verify.
\end{proof}

\begin{corollary}\label{finitecasecor}
Theorem \ref{infiniteeckmannhiltontheorem} holds when $K$ is finite.
\end{corollary}
\begin{proof}
Suppose $K$ is finite, $\mathscr{R}=\{R_k\mid k\in K\}$ and $\mathscr{S}=\{S_k\mid k\in K\}$ are $n$-domains, and $\{f_k\}_{k\in K}$ is a $K$-sequence in $\Omega^{n}(X,x)$. Applying Lemma \ref{eckmannhiltonlemma} to $\scrr$ and $\scrs$ in the case when $\phi$ is the identity gives $\prod_{\mathscr{R}}f_k\simeq\prod_{k=1}^{m}f_{k}\simeq\prod_{\mathscr{S}}f_k$ where both homotopies have image in $\bigcup_{k\in K}Im(f_k)$. 

For the second statement of Theorem \ref{infiniteeckmannhiltontheorem}, fix $n$-domains $\scrr=\{R_1,R_2,\dots R_m\}$ and $\scrs=\{S_1,S_2,\dots, S_m\}$ such that $F=\{k\in K\mid R_k=S_k\}$ is non-empty and $I^n\backslash \bigcup_{k\in F}R_k$ is path connected. Let $K=\{1,2,\dots,m\}\backslash F$ and write $K=\{k_1,k_2,\dots ,k_p\}$. Using Lemma \ref{shrinkingcubelemma} to shrink the $n$-cubes $R_k,S_k$, $k\in K$, we may assume that for all $k,k'\in K$, we have $R_k\cap S_{k'}=\emptyset$. Let $\scrr_0=\scrr$. If we have defined $n$-domain $\scrr_{i-1}=\{R_{1}^{i-1},R_{2}^{i-1},\dots R_{m}^{i-1}\}$, set $R_{k_i}^{i}=S_{k_i}$ and $R_{k}^{i}=R_{k}^{i-1}$ if $k\neq k_i$. Define $\scrr_{i}=\{R_{1}^{i},R_{2}^{i},\dots ,R_{m}^{i}\}$. We continue recursively until we reach $\scrr_{p}=\scrs$. Our application of Lemma \ref{shrinkingcubelemma} ensures that $\scrr^{i}$ is a well-defined $n$-domain for all $1\leq i\leq p$. Lemma \ref{twocyclelemma} gives homotopies
\[\prod_{\scrr}f_k=\prod_{\scrr_0}f_k\simeq \prod_{\scrr_1}f_k\simeq \prod_{\scrr_2}f_k\simeq\,\cdots \,\simeq\prod_{\scrr_p}f_k=\prod_{\scrs}f_k,\] each of which may be chosen to be constant on $\bigcup_{k\in F}R_k=\bigcup_{k\in F}S_k$.\end{proof}

\begin{remark}\label{rigidityremark}
We note that the homotopy $H:I^n\times I\to X$ used to prove Corollary \ref{finitecasecor} is ``rigid" in the sense that each cube $R_k$ in $I^n\times \{0\}$ is moved isometrically through a path of $n$-cubes to $S_k$ in $I^n\times \{1\}$. More precisely, for each $k$, there is an isotopy $G_k:I^n\times I\to I^n$ such that for each $t\in \ui$, there is an $n$-cube $A_t$ with $G_k(s,t)=L_{I^n,A_t}(s)$ and where $A_0=R_k$ and $A_1=S_k$. Then if $s\in A_t$, we have $H(s,t)=f_k(L_{A_t,I^n}(s))$ and $H$ is constant at $x\in X$ otherwise.
\end{remark}

\begin{remark}\label{temptingremark}
It is tempting to think that we might be able prove the first statement of Theorem \ref{infiniteeckmannhiltontheorem} by performing a sequence of homotopies similar to that in the proof of Lemma \ref{eckmannhiltonlemma}. However, this approach fails at the first step where we wish to fix a dimension $i\in\{1,2,\dots, n\}$ and create a sub-$n$-domain $\mathscr{A}$ of $\mathscr{R}=\{R_1,R_2,R_3,\dots\}$ where the $i$-th projections $p_i:I^n\to I$ map the interiors of each element of $\mathscr{A}$ to pairwise disjoint intervals. Indeed, $\mathscr{R}$ might have the property that for every $i\in \{1,2,\dots,n\}$, rational number $q\in I$, and $\epsilon>0$, there exists $R_k\in\scrr$ such that $p_i(R_k)\subseteq [q-\epsilon,q+\epsilon]$. Such an $n$-domain is straightforward to construct (See Figure \ref{baddomain}). If $\mathscr{R}$ has this property and $S_1\subseteq R_1$ is any sub-$n$-cube, then for any given $i\in \{1,2,\dots,n\}$, the open interval $int(p_i(S_1))$ contains a rational $q$ and therefore will contain $p_i(R_k)$ (and the $i$-th projection of any sub-$n$-cube $S_k\subseteq R_k$) for some $k>1$. For such an $n$-domain $\scrr$, it is not possible to make pairwise disjoint selections in any dimension.
\end{remark}
\begin{figure}[H]
\centering \includegraphics[height=2.3in]{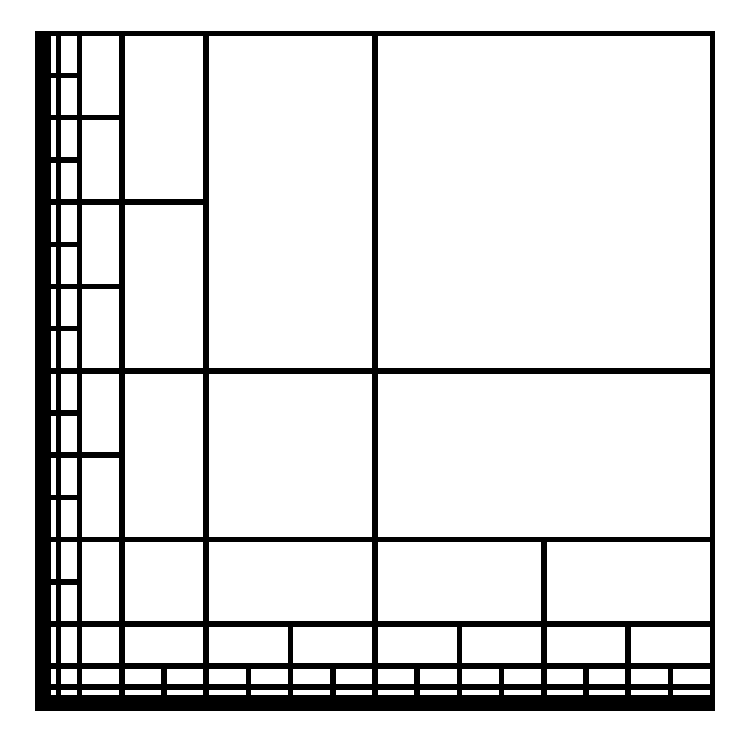}
\caption{\label{baddomain} A $2$-domain for which the proof of Lemma \ref{eckmannhiltonlemma} does not extend to the infinite case.}
\end{figure}
\section{The Infinite Case}\label{sectioninfinitecase}
\begin{definition}
If $R=\prod_{i=1}^{n}[a_i,b_i]$ is an $n$-cube in $(0,1)^n$, let $\mcc(R)$ denote the set of $n$-cubes in $I^n$ of the form $\prod_{i=1}^{n}A_i$ where $A_i\in \{[0,a_i],[a_i,b_i],[b_i,1]\}$. We give $\mcc(R)=\{C_1,C_2,\dots, C_{3^n}\}$ the natural lexicographical ordering inherited from $I^n$ so that the central cube is $C_{(3^n+1)/2}=R$ and $\mcc(R)$ forms a cubical decomposition of $I^n$.
\end{definition}
Notice that $\mcc(R)$ is obtained by subdividing $I^n$ using the hyperplane extensions of the $(n-1)$-dimensional faces of $\partial R$. 
\begin{definition}
An $n$-domain $\mathscr{R}=\{R_k\mid k\in K\}$ is \textit{proper at position $k_0\in K$} if $R_{k_0}\subseteq (0,1)^n$ and if for every $k\neq k_{0}$, we have $R_{k}\subseteq C$ for some $C\in \mcc(R_{k_0})$.
\end{definition}

\begin{remark}\label{propersubdomainremark}
Suppose $\mathscr{R}=\{R_k \mid k\in K\}$ is any $n$-domain and $R_{k_0}\subseteq (0,1)^n$. It is easy to see that there exists a sub-$n$-domain $\mathscr{S}=\{S_k\mid k\in K\}$ of $\mathscr{R}$ that is proper at position $k_0$. In particular, we set $S_k=R_k$ if $R_k$ already lies entirely in some element of $\mcc(R_{k_0})$, including $k=k_0$. If $R_k$ has non-empty intersection with two or more elements of $\mcc(R_{k_0})$, then $k\neq k_0$ and we may choose $S_k$ to be any $n$-cube of the form $R_k\cap C$ where $C\in \mcc(R_{k_0})$.
\end{remark}
\begin{lemma}\label{ntwistlemma}
If $\mathscr{R}=\{R_k\mid k\in\bbn\}$ is an $n$-domain such that $R_1\subseteq (0,1)^n$ and $\{f_k\}_{k\in\bbn}$ is a null-sequence in $\Omega^n(X,x)$, then there exists an $n$-domain $\mathscr{S}=\{S_k\mid k\geq 2\}$ such that $\prod_{\mathscr{R}}f_k\simeq f_{1}\cdot \left( \prod_{\mathscr{S}}\{f_k\}_{k\geq 2}\right)$ by a homotopy in $\bigcup_{k\in\bbn}Im(f_k)$.
\end{lemma}
\begin{proof}
Applying Remark \ref{propersubdomainremark}, find a sub-$n$-domain $\mathscr{R}'=\{R_k'\mid k\in\bbn\}$ of $\mathscr{R}$, which is proper at $k=1$. In particular, we have $R_1'=R_1$. By Lemma \ref{shrinkingcubelemma}, there exists a cube-shrinking homotopy $\prod_{\mathscr{R}}f_k\simeq \prod_{\mathscr{R}'}f_k=g$ where the homotopy has image in $\bigcup_{k\in\bbn}Im(f_k)$. Since each $R_{k}'$ lies in some element of the cubical decomposition $\mcc(R_1)=\{C_1,C_2,\dots ,C_{3^n}\}$ of $I^n$, each map $g_j=g|_{C_j}\circ L_{I^n,C_j}$ is an element of $\Omega^{n}(X,x)$. Moreover, if $K_j=\{k\in \bbn\mid R_{k'}\subseteq C_j\}$, then $\mathscr{S}_j=\{L_{C_j,I^n}(R_{k}')\}_{k\in K_j}$ is an $n$-domain such that $g_j=\prod_{\mathscr{S}_j}\{f_k\}_{k\in K_j}$. In particular, $C_{(3^n+1)/2}=R_1$ and $g_{(3^n+1)/2}=f_1$.

Choose a bijection $\phi\in \Sigma_{3^n}$ such that $\phi(1)=(3^n+1)/2$. Applying Lemma \ref{eckmannhiltonlemma}, we have $g\simeq h=\prod_{j=1}^{3^n}g_{\phi(j)}=f_1\cdot g_{\phi(2)}\cdot g_{\phi(3)}\cdots g_{\phi(3^n)}$ by a homotopy with image in $\bigcup_{j=1}^{3^n}Im(g_j)=\bigcup_{k\in\bbn}Im(f_k)$. Let $A_j=\left[\frac{j-1}{3^n},\frac{j}{3^n}\right]\times I^{n-1}$ and for each $k\in\{2,3,4,\dots\}$, if $R_{k}'\subseteq C_j$, set $S_k'=L_{C_j,A_j}(R_{k}')$. Let $A_{\geq 2}=\bigcup_{j=2}^{3^n}A_j$ and define $h'=h|_{A_{\geq 2}}\circ L_{I^n,A_{j\geq 2}}:I^n\to X$ so that $h'$ is a first-coordinate reparameterization of $g_{\phi(2)}\cdot g_{\phi(3)}\cdots g_{\phi(3^n)}$. Finally, setting $S_k=L_{I,A_{\geq 2}}(S_k')$, we have the desired $n$-domain $\mathscr{S}=\{S_2,S_3,S_4,\dots\}$ for which $h'=\prod_{\mathscr{S}}\{f_k\}_{k\geq 2}$.
\end{proof}
\begin{definition}\label{concatenationdef}
The \textit{standard $n$-domain} is the $n$-domain $\scrr=\{R_k\mid k\in\bbn\}$ where $R_k=\left[\frac{k-1}{k},\frac{k}{k+1}\right]\times I^n$. The \textit{infinite concatenation} of a null-sequence $\{f_k\}_{k\in \bbn}$ in $\Omega^{n}(X,x)$ is the $\scrr$-concatenation $\prod_{\scrr}f_k$ where $\scrr$ is the standard $n$-domain. We will typically denote this map as $\prod_{k=1}^{\infty}f_k$.
\end{definition}
\begin{lemma}\label{eckhiltontechlemma}
If $\mathscr{R}=\{R_k\mid k\in\bbn\}$ is an infinite $n$-domain and $\{f_k\}_{k\in \bbn}$ is a null-sequence in $\Omega^n(X,x)$, then $\prod_{\mathscr{R}}f_k\simeq \prod_{k=1}^{\infty}f_k$ by a homotopy in $\bigcup_{k=1}^{\infty}Im(f_k)$.
\end{lemma}
\begin{proof}
Set $\scrr_0=\scrr$. Replacing each $R_k$ with a sub-$n$-cube $R_{k}'\subseteq int(R_k)$ and applying Lemma \ref{shrinkingcubelemma}, we may assume that $R_k\subseteq (0,1)^n$ for all $R_k\in\mathscr{R}_0$. By Lemma \ref{ntwistlemma}, there exists an $n$-domain $\mathscr{R}_1$ over $\{2,3,4,\dots\}$ and a homotopy $H_1:I^n\times I\to X$ from $f_{1}\cdot\left( \prod_{\scrr_1}\{f_k\}_{k\geq 2}\right)$ to $\prod_{\scrr_{0}}f_k$ with $Im(H_1)\subseteq \bigcup_{k\in\bbn}Im(f_k)$. Observe from the proof of Lemma \ref{ntwistlemma} that since $R_k\subseteq (0,1)^n$ for all $R_k\in\mathscr{R}$, we have $S\subseteq (0,1)^n$ for all $S\in \mathscr{R}_1$. Applying Lemma \ref{ntwistlemma} recursively, we obtain a sequence of $n$-domains $\scrr_1,\scrr_2,\scrr_3,\dots$ (where $\scrr_m$ is an $n$-domain over $\{m+1,m+2,m+3,\dots\}$) and a sequence of homotopies $H_1,H_2,H_3,\dots:I^n\times I\to X$ such that $H_m$ is a homotopy from $f_{m}\cdot\left( \prod_{\scrr_m}\{f_k\}_{k\geq m+1}\right)$ to $\prod_{\scrr_{m-1}}\{f_m\}_{k\geq m}$ with $Im(H_m)\subseteq \bigcup_{k\geq m}Im(f_k)$. This induction is possible due to our initial choice of $\scrr$ and since from the definition of the homotopy $H_m$ obtained using Lemma \ref{ntwistlemma} we have that ($\forall$ $S\in \scrr_{m-1}$, $S\subseteq (0,1)^n$) $\Rightarrow$ ($\forall$ $S\in \scrr_{m}$, $S\subseteq (0,1)^n$).

In order to construct a homotopy $H$ from $\prod_{k=1}^{\infty}f_k$ to $\prod_{\mathscr{R_0}}f_k$, we perform an appropriate gluing of the infinite sequence of homotopies $\{H_m\}_{m\in\bbn}$. Let $G_m:I^n\times I\to X$, $G(x,t)=f_m(x)$ be the constant homotopy of $f_m$ for all $m\in\bbn$. For simplicity, define the following subsets of $I^{n}\times I$:
\begin{itemize}
\item $A_m=\left[1-\frac{1}{m},1\right]\times I^{n-2}\times \left[\frac{1}{m+1},\frac{1}{m}\right]$,
\item $B_m=\left[1-\frac{1}{m},1-\frac{1}{m+1}\right]\times I^{n-2}\times \left[0,\frac{1}{m+1}\right]$,
\item $C=\{1\}\times I^{n-2}\times \{0\}$.
\end{itemize}
Notice that $I^n\times I=\bigcup_{n\in\bbn}A_n\cup \bigcup_{n\in\bbn}\left(B_n\times \left[0,\frac{1}{m+1}\right]\right)\cup C$. We define $H$ by setting:
\begin{itemize}
\item $H|_{A_m}\circ L_{I^{n+1},A_m}=H_m$,
\item $H|_{B_m}\circ L_{I^{n+1},B_m}=G_m$,
\item $H(C)=x$.
\end{itemize}
It is straightforward to check that $H$ is well-defined, $H(a,0)=\left(\prod_{k=1}^{\infty}f_k\right)(a)$, and $H(a,1)=\left(\prod_{\mathscr{R_0}}f_k\right)(a)$. Since each $n$-cube in the set $\{A_m,B_m\mid m\in\bbn\}$ intersects at most finitely others from this same set, $H$ is clearly continuous at all points in $I^{n+1}\backslash C$. To check continuity at the points of $C$, let $U$ be an open neighborhood of $x$. Find an $M$ such that $Im(f_m)\subseteq U$ for all $m\geq M$. Then we have $Im(G_m)=Im(f_m)\subseteq U$ and $Im(H_m)\subseteq \bigcup_{k\geq m}Im(f_k)\subseteq U$ all $m\geq M$. Then $C\cup \bigcup_{m\geq M}(A_m\cup B_m)$ contains a neighborhood $V$ of $C$, and thus $f(V)\subseteq U$, completing the proof.
\end{proof}
\begin{figure}[H]
\centering \includegraphics[height=2.5in]{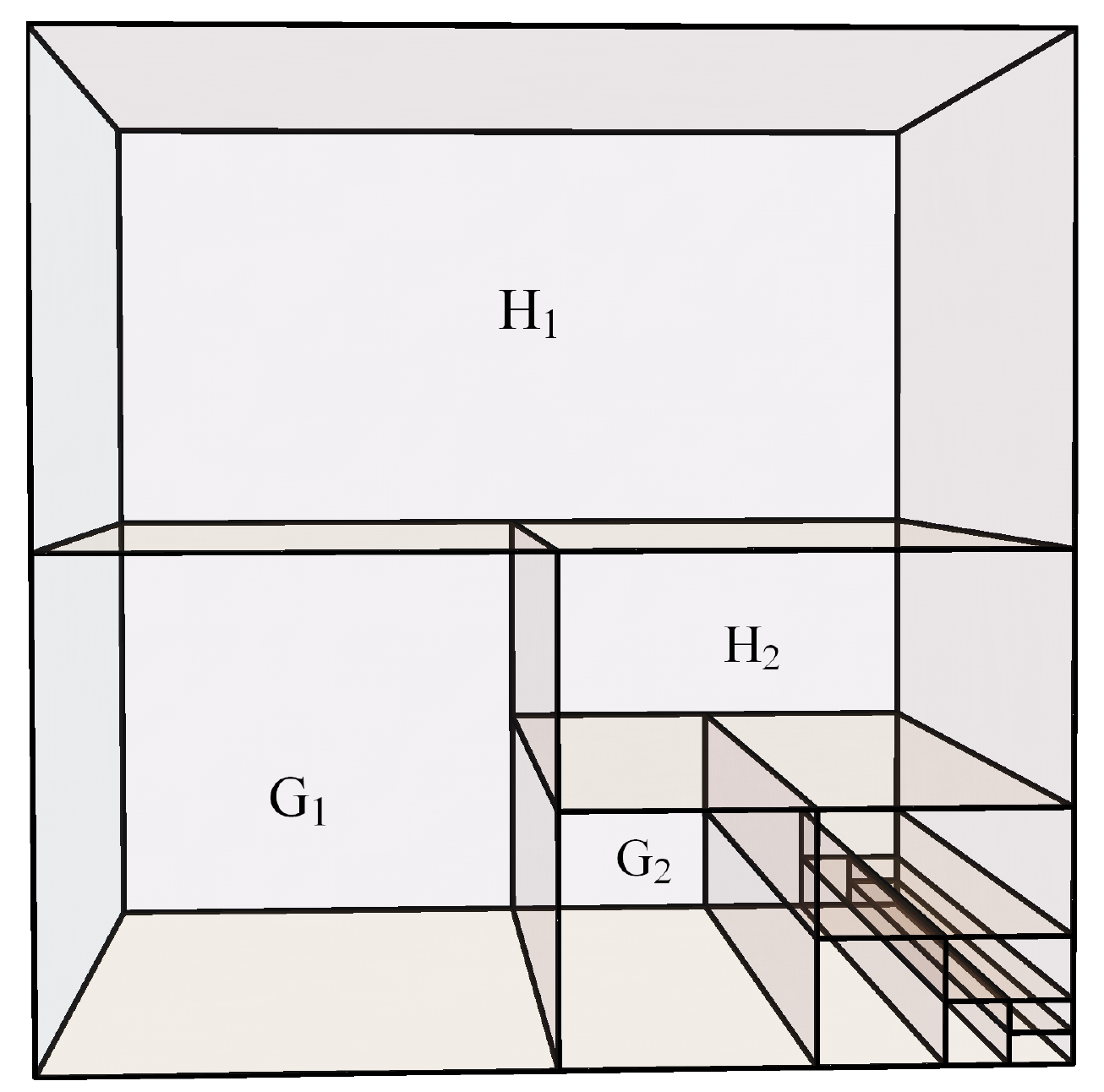}
\caption{\label{domain3} The homotopy $H$ constructed in the proof of Lemma \ref{eckhiltontechlemma} in the case $n=2$. The top of the cube is the domain of $\prod_{\mathscr{R}}f_k$. The homotopy $H_1$ moves $R_1$ to the top of the cubical domain of $G_1$, which is the constant homotopy of $f_1$. Similarly $H_2$ shuffles the image of $R_2$ under $H_1$ to the top of the domain of $G_2$ and so on.}
\end{figure}
\begin{proof}[Proof of Theorem \ref{infiniteeckmannhiltontheorem}]
The finite case is given by Corollary \ref{finitecasecor}. Supposing $K$ is infinite, we may assume without loss of generality that $K=\bbn$.

\textbf{Unrestricted Case:} Suppose $\mathscr{R}=\{R_1,R_2,R_3,\dots\}$ and $\mathscr{S}=\{S_1,S_2,S_3,\dots\}$ are infinite $n$-domains and $\{f_k\}_{k\in\bbn}$ is a null-sequence in $\Omega^n(X,x)$. Applying Lemma \ref{infiniteeckmannhiltontheorem} twice gives  $\left[\prod_{\mathscr{R}}f_k\right]=\left[ \prod_{k=1}^{\infty}f_k\right]=\left[\prod_{\mathscr{S}}f_k\right]$.

\textbf{Restricted Case:} For the second statement of Theorem \ref{infiniteeckmannhiltontheorem}, suppose that $F$ is a non-empty finite subset of $\bbn$ such that $S_k=R_k$ for all $k\in F$ and such that $I^n\backslash \bigcup_{k\in F}R_k$ is path connected. Let $K=\bbn\backslash F$ and put $\alpha=\prod_{\scrr}f_k$ and $\beta=\prod_{\scrs}f_k$. We break the remainder of the proof into three steps.

\textit{Step 1:} Fix a cubical decomposition $\scrc=\{C_1,C_2,C_3,\dots ,C_m\}$ of $I^n$ such that $\{R_k\mid k\in F\}\subseteq \scrc$. The $n$-domain $\scrr$ may form an infinite cubical decomposition of $I^n$ and so we must first work to replace $\scrc$ with a finer cubical decomposition, one element of which meets none of the $R_k$. Let $F'=\{j\in \{1,2,\dots ,m\}\mid C_j=R_k\text{ for }k\in F\}$. Additionally, fix some $k_0\in K$ and find an ``auxiliary" $n$-cube $A\subseteq int(R_{k_0})$. If there exists $k_1\in K$ such that $int(R_{k_0})\cap int(S_{k_1})\neq\emptyset$, fix this $k_1$ and choose $A\subseteq int(R_{k_0})\cap int(S_{k_1})$

We replace $\scrr$ and $\scrs$ with convenient sub-$n$-domains: Let $S_k'=R_k'=R_k=S_k$ if $k\in F$ and if $k\in K$, find $n$-cubes $R_k'$ and $S_k'$ such that
\begin{enumerate}
\item $R_k'\subseteq int(R_k)$ and $S_k'\subseteq int(S_k)$,
\item $R_k'$ and $S_k'$ lie entirely within some element of $\scrc$,
\item $(R_k'\cup S_k')\cap A=\emptyset$.
\end{enumerate}
It is easy to choose $R_k'$ and $S_k'$ so that (1) and (2) are satisfied. Once this is done, (3) is achieved by altering the choice of $R_{k_0}'$ so that it lies within $int(R_{k_0})\backslash A$. If $k_1$ was fixed such that $A\subseteq int(R_{k_0})\cap int(S_{k_1})$, then we alter the choice of $S_{k_1}'$ so that it lies within $int(S_{k_1})\backslash A$. By Lemma \ref{shrinkingcubelemma}, we may replace $\scrr$ with $\{R_k'\mid k\in\bbn\}$ and $\scrs$ with $\{S_k'\mid k\in \bbn\}$. Applying this replacement, we may assume that 
\begin{enumerate}
\item when $k\in F$, we have $R_k=S_k=C_{j(k)}$ for some $j(k)\in \{1,2,\dots ,m\}$,
\item when $k\in K$, we have $R_k\subseteq C_{j(k)}$ for some $j(k)\in \{1,2,\dots ,m\}$,
\item $A$ is disjoint from $\bigcup_{k\in \bbn}R_k\cup \bigcup_{k\in \bbn}S_k$.
\end{enumerate} 
Let $\scrc '$ be a cubical decomposition of $I^n$ so that $\{A\}\cup \{R_k\mid k\in F\}\subseteq \scrc '$. Applying Lemma \ref{shrinkingcubelemma} in the same fashion as above, we may replace $\scrr$ and $\scrs$ with sub-$n$-domains to ensure that $R_k=S_k\in \scrc '$ when $k\in F$ and so that for every $k\in K$, each of $R_k$ and $S_k$ lies within some element of $\scrc '$. Replacing $\scrc$ with $\scrc '$ allows us to assume that $A\in \scrc$ and that for each $k\in K$, $R_k$ and $S_k$ each lie in some element (possibly distinct) of $\scrc\backslash (\{A\}\cup \{R_k\mid k\in F\})$.

\textit{Step 2:} Using the cubical decomposition $\scrc$ obtained from Step 1, write $\scrc=\{C_1,C_2,\dots, C_m\}$ and define $F'$ as before. There exists unique $j_0\in \{1,2,\dots,m\}$ such that $C_{j_0}=A$. Let $G=\{j\in \{1,2,\dots ,m\}\mid j\notin (F'\cup \{j_0\})\}$. Define a finite $n$-domain $\scrd=\{D_1,D_2,\dots ,D_m\}$ so that $D_j=R_k$ if $j\in F'$ and $C_j=R_k$ and $D_j\subseteq A$ if $j\in G$. Define $\alpha_j=\alpha\circ L_{I^n,C_j}$ and $\beta_j=\beta\circ L_{I^n,C_j}$ so that $\alpha=\prod_{\scrc}\alpha_j$ and $\beta=\prod_{\scrc}\beta_j$. By Corollary \ref{finitecasecor}, there are homotopies $\prod_{\scrc}\alpha_j\simeq \prod_{\scrd}\alpha_j$ and $\prod_{\scrc}\beta_j\simeq \prod_{\scrd}\beta_j$, which are constant on $\bigcup_{j\in F'}D_j=\bigcup_{k\in F}R_k$. Therefore, it suffices to show that $\prod_{\scrd}\alpha_j$ and $\prod_{\scrd}\beta_j$ are homotopic by a homotopy that is constant on $\bigcup_{j\in F'}C_j$.

\textit{Step 3:} The maps $\prod_{\scrd}\alpha_j$ and $\prod_{\scrd}\beta_j$ agree on $I^n\backslash int(A)$, which contains $\bigcup_{j\in F'}C_j$. Therefore, it suffices to show that $\prod_{\scrd}\alpha_j|_{A}\circ L_{I^n,A}$ and $\prod_{\scrd}\beta_j|_{A}\circ L_{I^n,A}$ are homotopic. However, this follows from the from the unrestricted case and the fact that the homotopies used are constructed by isotopic rearrangements (recall Remark \ref{rigidityremark}). Specifically, if $k\in K$ and we have $R_k\subset C_p$ and $S_k\subseteq C_q$, then we may take $R_k'=L_{C_p,D_p}(R_k)$ and $S_k'=L_{C_q,D_q}(S_k)$. Setting $\scrr '=\{R_k'\mid k\in K\}$ and $\scrs '=\{S_k'\mid k\in K\}$, we have $\prod_{\scrd}\alpha_j|_{A}\circ L_{I^n,A}= \prod_{\scrr '}\{f_k\}_{k\in K}$ and $\prod_{\scrd}\beta_j|_{A}\circ L_{I^n,A}= \prod_{\scrs '}\{f_k\}_{k\in K}$. The infinite unrestricted case gives $\prod_{\scrr '}\{f_k\}_{k\in K}\simeq \prod_{\scrs '}\{f_k\}_{k\in K}$, completing the proof.
\end{proof}
\begin{remark}
The proof of Theorem \ref{infiniteeckmannhiltontheorem} also works if $n$-cubes are replaced with arbitrary $n$-dimensional convex sets in $I^n$ and the definitions of $n$-domain and $\scrr$-concatenation are altered accordingly. 
\end{remark}
\section{Some immediate consequences of Theorem \ref{infiniteeckmannhiltontheorem}}\label{sectionconsequences}

The following corollaries to Theorem \ref{infiniteeckmannhiltontheorem} state that infinite $n$-loop concatenations (recall Definition \ref{concatenationdef}) may be rearranged in various ways without altering the homotopy class of the product. As mentioned in the introduction, the author has learned that some of these consequences appear in \cite{Kawamurasuspensions}. For instance, Lemma \ref{splitlemma} appears as \cite[Theorem 2.1]{Kawamurasuspensions}, one of the main results in that paper.

\begin{corollary}
For any null-sequence $\{f_k\}_{k\in\bbn}$ in $\Omega^{n}(X,x)$ and $\phi\in\Sigma_{\bbn}$, we have $\left[\prod_{k=1}^{\infty}f_k\right]=\left[\prod_{k=1}^{\infty}f_{\phi(k)}\right]$.
\end{corollary}
\begin{proof}
Let $\mathscr{R}=\{R_1,R_2,R_3,\dots\}$ be the standard $n$-domain so that $\prod_{\mathscr{R}}f_k=\prod_{k=1}^{\infty}f_k$ and $\prod_{\mathscr{R}}f_{\phi(k)}=\prod_{k=1}^{\infty}f_{\phi(k)}$. Notice that for $n$-domain \[\scrs=\{R_{\phi(1)},R_{\phi(2)},R_{\phi(3)},\dots\},\] we have $\prod_{\mathscr{S}}f_{\phi(k)}=\prod_{k=1}^{\infty}f_k$. Hence, $\left[\prod_{k=1}^{\infty}f_k\right]=\left[\prod_{\mathscr{S}}f_{\phi(k)}\right]= \left[\prod_{\mathscr{R}}f_{\phi(k)}\right]=\left[\prod_{k=1}^{\infty}f_{\phi(k)}\right]$ where the second equality is acquired from applying Theorem \ref{infiniteeckmannhiltontheorem}.
\end{proof}
\begin{corollary}\label{doubleproductshuffle}
If $\{f_k\}_{k\in\bbn}$ and $\{g_k\}_{k\in\bbn}$ are null-sequences in $ \Omega^{n}(X,x)$, then $\left[\prod_{k=1}^{\infty}(f_k\cdot g_k)\right]=\left[\prod_{k=1}^{\infty}f_k\right]+\left[\prod_{k=1}^{\infty}g_k\right]$.
\end{corollary}

\begin{proof}
Let $a_k$ be the midpoint of $\left[\frac{k-1}{k},\frac{k}{k+1}\right]$ and define $A_{k,1}=\left[\frac{k-1}{k},a_k\right]\times I^{n-1}$, $A_{k,2}=\left[a_k,\frac{k}{k+1}\right]\times I^{n-1}$, and $R_k=A_{k,1}\cup A_{k,2}$. Let \[\mathscr{R}=\{A_{1,1},A_{1,2},A_{2,1},A_{2,2},A_{3,1},A_{3,2},\dots\}\] so that $\prod_{\scrr}\{f_1,g_1,f_2,g_2,\dots\}=\prod_{k=1}^{\infty}(f_k\cdot g_k)$. Define $B_k=L_{I^n,[0,1/2]\times I^{n-1}}(R_k)$, $C_k=L_{I^n,[1/2,1]\times I^{n-1}}(R_k)$, and consider the $n$-domain $\scrs=\{B_1,C_1,B_2,C_2,B_3,C_3,\dots\}$ so that $\prod_{\scrs}\{f_1,g_1,f_2,g_2,\dots\}=\left(\prod_{k=1}^{\infty}f_k\right)\cdot \left(\prod_{k=1}^{\infty}g_k\right)$. Overall, we have 
\begin{eqnarray*}
\left[\prod_{k=1}^{\infty}f_k\right]+\left[\prod_{k=1}^{\infty}g_k\right] &=&  \left[\left(\prod_{k=1}^{\infty}f_k\right)\cdot \left(\prod_{k=1}^{\infty}g_k\right)\right]\\
&=&
\left[\prod_{\scrs}\{f_1,g_1,f_2,g_2,\dots\}\right]\\
&=& \left[\prod_{\scrr}\{f_1,g_1,f_2,g_2,\dots\}\right]\\
&=&
\left[\prod_{k=1}^{\infty}(f_k\cdot g_k)\right]
\end{eqnarray*}
where the third equality is acquired using Theorem \ref{infiniteeckmannhiltontheorem}.
\end{proof}

For an infinite sequence $\{(X_m,x_m)\}_{m\in\bbn}$ of path-connected based spaces, let $\wt{\bigvee}_{m}X_m$ denote the wedge sum, i.e. one-point union, where the basepoints $x_m$ are identified to a canonical basepoint $x_0\in\wt{\bigvee}_{m}X_m$. We give $\wt{\bigvee}_{m}X_m$ the topology consisting of sets $U$ such that $U\cap X_m$ is open in $X_m$ for all $m\in\bbn$ and if $x_0\in U$, then $X_m\subseteq U$ for all but finitely many $m$. We call $\wt{\bigvee}_{m}X_m$ the \textit{shrinking wedge} of the sequence $\{(X_m,x_m)\}_{m\in\bbn}$. 
Note that there is a canonical embedding $\sigma:\wt{\bigvee}_{m}X_m\to \prod_{m\in\bbn}X_m$ into the infinite direct product. For $k\geq 0$, if $X_m=S^k\subseteq \bbr^{k+1}$ is the $k$-dimensional unit sphere, with basepoint $x_m=(1,0,\dots ,0)$, then $\bbh_k=\wt{\bigvee}_{m}X_m$ is the \textit{$k$-dimensional Hawaiian earring} space, which embeds in $\bbr^{k+1}$.

A based space $(X_m,x_m)$ is \textit{well-pointed} if the inclusion $\{x_m\}\to X_m$ is a cofibration. Standard arguments in homotopy theory give that $X_m$ is homotopy equivalent (rel. basepoint) to the space $X_{m}^+$ obtained by attaching an interval to $X_m$ at $x_m$ by $1\sim x_m$ and taking the image $x_{m}^+$ of $0$ (at which $X_m$ is locally contractible) to be the basepoint. In this case, it is straightforward to prove that $\wt{\bigvee}_{m}(X_m,x_m)\simeq \wt{\bigvee}_{m}(X_{m}^+,x_{m}^+)$ rel. basepoint. Hence, by assuming that each $(X_m,x_m)$ is well-pointed, we may assume that we are using a shrinking wedge $\wt{\bigvee}_{m}X_{m}^+$ to which the results of \cite{EK00higher} apply.


\begin{lemma}\label{splitlemma}\cite[Theorem 2.1]{Kawamurasuspensions}
For any sequence $\{(X_m,x_m)\}_{m\in\bbn}$ of path-connected spaces and integer $n\geq 2$, the homomorphism $\sigma_{\#}:\pi_n\left(\wt{\bigvee}_{m}X_m\right)\to \pi_n\left(\prod_{m}X_m\right)$ splits naturally.
\end{lemma}

\begin{proof}
Fix $n\geq 2$ and identify $\pi_n\left(\prod_{m}X_m,x_0\right)$ with $\prod_{m}\pi_n(X_m,x_m)$. We define a section $s:\prod_{m}\pi_n(X_m,x_m)\to \pi_n\left(\wt{\bigvee}_{m}X_m\right)$ of $\sigma_{\#}$. Consider a sequence of maps $f_m\in \Omega^{n}(X_m,x_m)$, $m\in\bbn$. Since $X_m\subseteq \wt{\bigvee}_{m}X_m$, we may form the infinite concatenation $\prod_{m=1}^{\infty}f_m$. Define $s:\pi_n(\prod_{m}X_m,(x_m))\to \pi_n(\wt{\bigvee}_{m}X_m,x_0)$ by $s([f_1],[f_2],[f_3],\dots)=\left[\prod_{m=1}^{\infty}f_m\right]$. 

To see that $s$ is well-defined consider another sequence of maps $g_m\in \Omega^{n}(X_m,x_m)$ such that $f_m\simeq g_m$ for all $m\in\bbn$. Let $H_m:I^n\times I\to X_m$ be a homotopy from $f_m$ to $g_m$. Let $\scra=\{A_1,A_2,A_3,\dots\}$ be the standard $n$-domain and define a homotopy $H:I^n\times I\to X$ so that $H|_{A_m\times I}=H_m\circ L_{I^{n+1},A_m\times I}$ and $H(\{1\}\times I^{n}\times I)=x_0$. Since every neighborhood $U$ of $x_0$ contains $X_m$ for all but finitely many $m$, $U$ also contains $Im(H_m)$ for all but finitely many $m$. Thus $H$ is continuous and gives a homotopy from $\prod_{m=1}^{\infty}f_m$ to $\prod_{m=1}^{\infty}g_m$. Therefore, $s([f_1],[f_2],[f_3],\dots)=s([g_1],[g_2],[g_3],\dots)$, proving $s$ is well-defined. If $p_m:\prod_{m}X_m\to X_m$ is the projection, then $p_m\circ \prod_{m=1}^{\infty}f_m\simeq f_m$ making it clear that $\sigma_{\#}\circ s=id$.

Next, we check that $s$ is a homomorphism. Consider elements $([f_1],[f_2],\dots)$ and $([g_1],[g_2],\dots)$ in $\prod_{m}\pi_n(X_m,x_m)$. We have
\begin{eqnarray*}
s([f_1],[f_2],\dots)+s([g_1],[g_2],\dots) &=& \left[\prod_{m=1}^{\infty}f_m\right]+ \left[\prod_{m=1}^{\infty}g_m\right]\\
&=&  \left[\left(\prod_{m=1}^{\infty}f_m\right)\cdot\left(\prod_{m=1}^{\infty}g_m\right)\right]\\
&=& \left[\prod_{m=1}^{\infty}(f_m\cdot g_m)\right]\\
&=& s([f_1\cdot g_1],[f_2\cdot g_2],\dots)\\
&=&  s(([f_1],[f_2],\dots )+([ g_1],[ g_2],\dots))
\end{eqnarray*}
where the third equality follows from Corollary \ref{doubleproductshuffle}.
\end{proof}

\begin{corollary}\label{splittingcorollary}\cite[Corollary 2.2]{Kawamurasuspensions}
For any sequence $\{(X_m,x_m)\}_{m\in\bbn}$ of path-connected spaces and integer $n\geq 2$, there is a natural isomorphism \[\pi_n\left(\wt{\bigvee}_{m}X_m\right)\cong \pi_{n+1}\left(\prod_{m}X_m,\wt{\bigvee}_{m}X_m\right)\oplus \prod_{m}\pi_n(X_m).\]
\end{corollary}

\begin{proof}
By Lemma \ref{splitlemma}, the inclusion $\sigma:\wt{\bigvee}_{m}X_m\to \prod_{m}X_m$ induces a surjection in every dimension $n\geq 2$ and thus the homotopy exact sequence of the pair $\left(\prod_{m}X_m,\wt{\bigvee}_{m}X_m\right)$ \cite[Chapter IV.2]{WhiteheadEOH} splits into short exact sequences for $n\geq 2$.
{\scriptsize{\[\xymatrix{
0 \ar[r] &  \pi_{n+1}\left(\prod_{m}X_m,\wt{\bigvee}_{m}X_m\right) \ar[r]^-{\partial} & 
 \pi_n\left(\wt{\bigvee}_{m}X_m\right) \ar@{->>}[r]^-{\sigma_{\#}}  & \pi_n\left( \prod_{m}X_m\right) \ar[r] & 0
}\]}}
Since Lemma \ref{splitlemma} also ensures that $\sigma_{\#}$ splits naturally, the corollary follows.
\end{proof}

\begin{remark}[Topological homotopy groups]\label{topologicalhomotopyremark}
There is a function $p:\prod_{m}\Omega^{n}(X_m,x_m)\to \Omega^{n}(\bigvee_{m}X_m,x_0)$ given by $p(f_1,f_2,f_3,\dots)=\prod_{m=1}^{\infty}f_m$. The continuity of $f$ follows easily from the fact that $\prod_{m}\Omega^{n}(X_m,(x_m))$ has the product topology and $\wt{\bigvee}_{m}X_m$ has the shrinking wedge topology. If $h:\Omega^n(\prod_{m}X_m,(x_m))\to \prod_{m}\Omega^{n}(X_m,x_m)$ is the canonical homeomorphism, then the map $p\circ h$ is continuous and induces the homomorphism $s$ in Lemma \ref{splitlemma}. Hence, the splitting homomorphism $s$ is continuous when the homotopy groups are given the natural quotient topology inherited from the compact-open topology. For instance, if $X_m$ is a sequence of $(n-1)$-connected CW-complexes, then the results of \cite{EK00higher} ensure that $\pi_{n+1}(\prod_{m}X_m,\wt{\bigvee}_{m}X_m)=0$ and thus $\pi_n(\wt{\bigvee}_{m}X_m)\cong \prod_{m}\pi_n(X_m)$ is, in fact, a topological isomorphism of an infinite direct product of the discrete groups $\pi_n(X_m)$.
\end{remark}

\begin{example}
For the $k$-dimensional Hawaiian earring $\bbh_k$, we have $\pi_n\left(\bbh_{k}\right)\cong \pi_{n+1}\left(\prod_{m}S^k,\bbh_{k}\right)\oplus \prod_{m}\pi_n(S^k)$ for all $n\geq 2$. In particular, this means that 
The main result of \cite{EK00higher} ensures that when $k=n\geq 2$, we have $\pi_{n+1}\left(\prod_{m}S^n,\bbh_{n}\right)=0$. In the case $n=3$ and $k=2$, $\pi_3(S^2)=\bbz$ is generated by the Hopf fibration. Hence, as pointed out in \cite[Corollary 2.2]{Kawamurasuspensions},  $\pi_3\left(\bbh_{2}\right)\cong \pi_{4}\left(\prod_{m}S^2,\bbh_{2}\right)\oplus \bbz^{\bbn}$, where the isomorphism type of $\pi_{4}\left(\prod_{m}S^2,\bbh_{2}\right)$ remains unknown since classical cellular approximation and excision arguments fail to apply.
\end{example}

\begin{remark}\label{operadremark}
All of the technical work required to Prove Theorem \ref{infiniteeckmannhiltontheorem} is carried out in the domain $I^n$ and allows one to consider the following infinitary extension of the little $n$-cubes operad \cite{BV} used for the $n$-loop-space recognition principle \cite{MayGoILS}. For any $m\in\bbn$, consider the space $\mcc_n(m)$ of $n$-domains $\scrr=\{R_1,R_2,\dots,R_m\}$ indexed by $\{1,2,\dots,m\}$. Formally, let $M(I^n,I^n)$ denote the space of self maps of $I^n$ with the compact-open topology. Then we take $\mcc_n(m)$ to be the subspace of $\prod_{k=1}^{m}M(I^n,I^n)$ consisting $m$-tuples $\scrr:=(L_{I^n,R_1},L_{I^n,R_2},\dots,L_{I^n,R_m})$ such that $int(R_k)\cap int(R_{k'})=\emptyset$ when $k\neq k'$. The $n$-disks operad acts on $\Omega^n(X,x)$ in the sense that for every $m$ there is an action $\mcc_n(m)\times\Omega^n(X,x)^m\to \Omega(X,x)$, given by $(\scrr,\{f_k\}_{k=1}^{m})\mapsto \prod_{\scrr}f_k$. 

Sending $n\to\infty$ gives the spaces $\mcc_{\infty}(m)=\varinjlim_{n}\mcc_n(m)$, which form an $E_{\infty}$-operad. We consider the situation where $m\to \infty$. Let $\mcc_n(\omega)$ denote the subspace of the direct product space $\prod_{k\in\omega}M(I^n,I^n)$ consisting of $\omega$-sequences $\scrr:=(L_{I^n,R_1},L_{I^n,R_2},L_{I^n,R_3},\dots)$ such that $int(R_k)\cap int(R_{k'})=\emptyset$ when $k\neq k'$. Let $\Omega^{n}_{null}(X,x)$ denote the set of null sequences $\{f_{k}\}_{k\in \omega}$ topologized as a subspace of the infinite direct product $\prod_{k\in\omega}\Omega^{n}(X,x)$. Then $\mcc_n(\omega)$ acts on the space of null-sequences by the infinitary action $\mcc_n(\omega)\times \Omega^{n}_{null}(X,x)\to \Omega^n(X,x)$ given by $(\scrr,\{f_k\}_{k\in\omega})\mapsto \prod_{\scrr}f_k$. A proof that this action is continuous requires a fairly tedious analysis of the respective direct product and compact-open topologies but is straightforward. Moreover, it becomes clear that one can extend the little $n$-cubes operad structure maps to include $\mcc_n(\omega)$ with corresponding structure maps such as $\mcc_n(2)\times \mcc_n(j)\times \mcc_n(\omega)\to \mcc_n(\omega)$ where \[(\{A_1,A_2\},\{B_1,B_2,\dots,B_j\},\{C_1,C_2,C_3,\dots\})\mapsto \{D_1,D_2,\dots,D_j,E_1,E_2,E_3,\dots \}\]
for $D_k=L_{I^n,A_1}(B_k)$ and $E_k=L_{I^n,A_2}(E_k)$. Finally, we recall the known fact that the space $\mcc_n(m)$ is $(n-2)$ connected for $n\geq 2$ \cite{MayGoILS}. The rigidity of the homotopies used to prove the unrestricted infinite case of Theorem \ref{infiniteeckmannhiltontheorem} implies that $\mcc_{n}(\omega)$ is path connected; whether or not $\mcc_{n}(\omega)$ is $(n-2)$-connected seems plausible but remains unproven.
\end{remark}

\end{document}